\documentclass[numbook,envcountsame,global]{svjour}

\usepackage[T1]{fontenc}
\usepackage{amsfonts}
\usepackage{amsmath}
\usepackage{amssymb}
\usepackage{amscd}

\newcommand{\be}{\begin{equation}}
\newcommand{\ee}{\end{equation}}
\newcommand{\bd}{\[}
\newcommand{\ed}{\]}
\newcommand{\Z}{\mathbb{Z}}

\def\è{\`{e}}
\def\à{\`{a}}
\def\ò{\`{o}}
\def\ù{\`{u}}
\def\ì{\`{i}}
\def\é{\'{e}}

\newcommand{\means}[1]{\hbox{$ [\kern -.4em [\, {#1}\, ]\kern -.4em]$}}
\DeclareMathOperator{\Aut}{Aut}

\begin{document}

\title{Circular nearrings: geometrical and combinatorial aspects}
\author{Anna Benini\inst{1}\and Achille Frigeri\inst{2} \and Fiorenza Morini\inst{3}}
\institute{Università di Brescia, \email{anna.benini@ing.unibs.it}\and Politecnico di Milano, \email{achille.frigeri@polimi.it}
\and Università di Parma, \email{fiorenza.morini@unipr.it}}

\maketitle


\begin{abstract}
Let $(N,\Phi)$ be a circular Ferrero pair. We define the disk with center $b$ and radius $a$, $\mathcal{D}(a;b)$, as
\bd
    \mathcal{D}(a;b)=\{x\in \Phi(r)+c\mid r\neq 0,\ b\in \Phi(r)+c,\ |(\Phi(r)+c)\cap (\Phi(a)+b)|=1\} .
\ed
We prove that in the field-generated case there are many analogies with the
Euclidean geometry. Moreover, if $\mathcal{B}^{\mathcal{D}}$ is the set of all disks, then, in some interesting cases, we show that the incidence structure $(N,\mathcal{B}^{\mathcal{D}},\in)$ is actually a balanced incomplete block design.

\end{abstract}


\section{Introduction}

Almost at the same time, but independently, Clay \cite{AC-68-1,AC-68-2} and Ferrero \cite{F-70} studied a new kind of algebraic structures, namely planar nearrings. It turns out that this class of nearrings is closely related with incidence geometry \cite{C-88-2}, combinatorics \cite{C-72}, coding theory \cite{F-90} and experimental designs \cite{P-83}.
An example in \cite{AC-68-1} proves that the circles of the usual Euclidean plane can be described in terms of the structural parts of a planar nearring whose additive group is $(\mathbb{C},+)$, leading the author to define the concept of circularity for planar nearrings. Moreover it is possible to construct double planar nearrings with a circular component as the following example shows.
\begin{example}\label{es1}
On the complex number field $(\mathbb{C},+,\cdot)$, for $a,b\in \mathbb{C}$, define two operations $*$ and $\circ$:
\bd
\begin{array}{l}
    a*b=|a|\cdot b,\\
    a\circ b=
    \begin{cases}
        0 & \text{if $a=0$}, \\
        \dfrac{a}{|a|}\cdot b & \text{if $a\neq 0$}.
    \end{cases}
\end{array}
\ed
Then $(\mathbb{C},+,*)$ and $(\mathbb{C},+,\circ)$ are planar nearrings and each of $*$ and $\circ$ is left distributive over the other, i.e., for any $a,b,c\in \mathbb{C}$ we have
\bd
    \begin{array}{l}
        a*(b\circ c)=(a*b)\circ (a*c), \\
        a\circ (b*c)=(a\circ b)*(a\circ c). \\
    \end{array}
\ed
Hence $(\mathbb{C},+,*,\circ)$ is a double planar nearring with a circular component.
\end{example}
As circularity proves to be interesting in application (see for example \cite{B-04}), in \cite{C-92} and \cite{C-94}, Clay proposes a \lq\lq tentative\rq\rq\ of definition of interior point of a circle (when a double planar nearring has a circular component). Let $(N,+,\star,\circ)$ be a double planar nearring and suppose $(N,+,\circ)$ is circular with circles $N^*\circ a+b$. Then $c\in N$ is an interior point for the circle $N^*\circ a+b$ if:
\begin{enumerate}
    \item $c\not \in N^* \circ a+b$;
    \item every ray $N\star a+b$ from $c$ intersects the circle $N^*\circ a+b$.
\end{enumerate}
Applied to Example \ref{es1} this definition works exactly as expected, namely it gives the interior points of a circle in the Euclidean plane. In \cite{E-97} the author proves some properties of the interior points, but no special combinatorial structure seems to arise from this approach. Moreover, the definition is not very practical, not only because it requires a great amount of calculation, but especially because it is not peculiar to the circular nearring. In fact the definition requires that a circular nearring has a planar \lq\lq partner\rq\rq, and in general there exist many possible partners each giving a different set of interior points for a given circle, as we prove in the Example \ref{z61}.\\
Our aim is to find a definition of interior point for a circle inherent to the structure of circular planar nearring without considering a planar partner, and to study its related geometry.

\section{Preliminaries and notations}
\subsection{Circular planar nearrings}

\begin{definition}\label{1.3}
A \emph{(left) nearring} is an algebraic structure $(N,+,\cdot)$ on a nonempty set $N$ with two inner operations, $+$ and $\cdot$, such that:
\begin{enumerate}
    \item $(N,+)$ is a group;
    \item $(N,\cdot )$ is a semigroup;
    \item the left distributive law holds, i.e.,
        \bd
            \forall x,y,z\in N,\ \ \ x\cdot (y+z)=(x\cdot y)+(x\cdot z).
        \ed
\end{enumerate}
Moreover, if $(N\setminus \{0\},\cdot )$ is a group, then $(N,+,\cdot)$ is a \emph{(left) nearfield}.
\end{definition}
We remind now some definitions and simple properties of nearrings, for details see \cite{C-92}.\\
Let $(N,+,\cdot )$ be a nearring, then for any $x,y\in N$, $x\cdot(-y)=-(x\cdot y)$ and $x\cdot 0=0$; moreover if for any $x\in N$ we have $x\cdot 0=0\cdot x=0$, $N$ is said to be \emph{0-symmetric}.\\
We say that $a,b\in N$ are \emph{equivalent multipliers} if and only if for all $n\in N$, $a\cdot n=b\cdot n$. It is easy to see that to be equivalent multipliers is an equivalence relation and we denote it by $\equiv_m$. We have the following fundamental definition.
\begin{definition} \label{1.7}
A nearring $(N,+,\cdot )$ is said to be \emph{planar} if:
\begin{enumerate}
    \item $|N/\equiv_m|\geq3$;
    \item $\forall a,b,c\in N$, with $a\not\equiv_m b$, the equation
        \bd
            a\cdot x=b\cdot x+c
        \ed
        has a unique solution in $N$.
\end{enumerate}
\end{definition}
For planar nearrings we consider the set $A=\{n\in N\mid n\equiv_m~ 0~\}$, called \emph{annihilator} of $N$, and we write
\bd
    N^0=N\setminus \{0\}\ \ \ \text{and} \ \ N^*=N\setminus A.
\ed
It is well known that planar nearrings are 0-symmetric; moreover planar nearrings with identity are also planar nearfield. Vice-versa a finite nearfield with at least three elements is planar and 0-symmetric and its additive group is the additive group of a field.
\begin{definition}\label{1.11}
Let $(N,+)$ be a group. A subgroup of automorphisms $\Phi$, $\{1\}\neq\Phi<\mathrm{\Aut}N$, is said to be \emph{regular} if for any $\varphi\in \Phi \setminus \{1\}$, $\varphi$ is a fixed point free (f.p.f.) automorphism, i.e., $\varphi(x)=x\Leftrightarrow x=0$. Moreover, if for any $\varphi \in \Phi \setminus \{1\}$, $-\varphi +1$ is surjective, the pair $(N,\Phi)$ is called a \emph{Ferrero pair}. Then an \emph{orbit} of $\Phi$ is the set
\bd
    \Phi(a)=\{\varphi(a) \mid \varphi \in \Phi\},
\ed
for some $a\in N$ ($\Phi(0)=\{0\}$ is the trivial orbit).
\end{definition}
The concept of Ferrero pair is central in this framework since every planar nearring can be constructed from such pair by the so called \emph{Ferrero Planar Nearring Factory} (for details, see \cite{C-92},\S4.1). Moreover we remember that, even if non-isomorphic nearrings can be generated by the same pair, the knowledge of the generating pair suffices in the study of the geometrical properties of the nearring.
\begin{theorem}[Ferrero, Clay] \label{1.13}
Let $(N,\Phi)$ be a finite Ferrero pair. Then
\begin{enumerate}
    \item $\forall a\in N^0,\ |\Phi(a)|=|\Phi|$;
    \item $\forall a\in N$, $\forall b\in \Phi(a)$, $\Phi(a)=\Phi(b)$;
    \item $\{\Phi(a)\mid a\in N\}$ is a partition of $N$;
    \item $|\Phi|$ divides $|N|-1$.
\end{enumerate}
\end{theorem}
\begin{definition} \label{1.15}
A planar nearring is said to be \emph{field-generated} if it is generated from a Ferrero pair $(N,\Phi)$, where $(N,+,\cdot )$ is a field and $\Phi$ is isomorphic to a subgroup of $(N^{0},\cdot )$.
\end{definition}
It follows immediately that a field-generated planar nearring has a very simple structure, in fact every non trivial orbit is isomorphic to a multiplicative subgroup of a field.\\
Now, on a planar nearring $(N,+,\cdot )$ we consider the incidence structure $(N,\mathcal{B}^*,\in )$ where $\mathcal{B}^*\subset \mathcal{P}(N)$ is  defined by
\bd
    \mathcal{B}^*=\{N^*\cdot a+b\mid a,b\in N,\ a\neq 0\}.
\ed
As usual in this contest, we call the elements of $N$ points and those of $\mathcal{B}^*$ blocks. It is immediate to observe that the structure $(N,\mathcal{B}^*,\in )$ depends only on the pair $(N,\Phi)$ in the sense that if $(N,+,*)$ and $(N,+,\circ )$ are two planar nearrings constructed from the same Ferrero pair, then they yield identical $(N,\mathcal{B}^*,\in )$; this allows the following definition.
\begin{definition} \label{1.17}
A planar nearring $(N,+,\cdot )$, or the pair $(N,\Phi)$ from which is generated, or the incidence structure $(N,\mathcal{B}^*,\in )$ yielded, is said \emph{circular} if every three distinct points of $N$ belong to at most one block of $\mathcal{B}^*$, and if every two distinct points belong to at least two distinct blocks. In this case the block $N^*\cdot a+b=\Phi(a)+b$ is called \emph{circle} with \emph{center} $b$ and \emph{radius} $a$ (see \cite{C-92},\S 5.1).
\end{definition}
\begin{definition} An algebraic structure  $(N,+,\ast, \circ )$ is a \emph{(left) double  planar nearring} if each of $(N,+,\ast)$ and $(N,+, \circ )$ is a
(left) planar nearring, and each of $\ast$ and $\circ$ is left distributive over the other.
\end{definition}
\begin{example} \label{z61}
Let consider the pairs $(\mathbb{Z}_{61},\Phi)$, $(\mathbb{Z}_{61},\Gamma)$ and $(\mathbb{Z}_{61},\Sigma)$, where $\Phi$, $\Gamma$ and $\Sigma$ are isomorphic to the multiplicative subgroups of $(\mathbb{Z}_{61}^0,\cdot)$ generated by $11$, $9$ and $13$, respectively. Each pair is circular  and moreover if $(\mathbb{Z}_{61},+,*_{\Phi})$, $(\mathbb{Z}_{61},+,*_{\Gamma})$ and $(\mathbb{Z}_{61},+,*_{\Sigma})$ are the nearrings they yield, then it is easy to prove that $(\mathbb{Z}_{61},+,*_{\Phi},*_{\Gamma})$ and $(\mathbb{Z}_{61},+,*_{\Phi},*_{\Sigma})$ are double planar nearrings. If we try to construct the set of interior points for the circle $\Phi(1)$ using the definition proposed in \cite{C-92}, for $(\mathbb{Z}_{61},+,*_{\Phi},*_{\Gamma}) $ we obtain the union of the orbits on $0,4,5,7,9,10,13,19,20$, while for $(\mathbb{Z}_{61},+,*_{\Phi},*_{\Sigma})$ we obtain the orbits on $0,3,4,5,7,8,13,14,15,19,20,25$.

\end{example}

\subsection{BIBDs}

\begin{definition}\label{1.18}
Let $(X,\mathcal{B},\in )$ be an incidence structure with $|X|=v$ and $|\mathcal{B}|=b$. If there exist two integers $k$ and $r$ so that for any $B\in \mathcal{B}$, $|B|=k$, and every $x\in X$ belongs to exactly $r$ distinct blocks $B_1,...,B_r\in \mathcal{B}$, then $(X,\mathcal{B},\in )$ is a \emph{tactical configuration} with parameters $v$, $b$, $k$, $r$. Moreover if an integer $\lambda$ exists so that every pair of points belongs to exactly $\lambda$ distinct blocks, then $(X,\mathcal{B},\in )$ is a \emph{balanced incomplete block design (BIBD)} with parameters $v$, $b$, $k$, $r$, $\lambda$. If $a,b\in X$, we use the notation $\means{a,b}$ to denote the number of blocks which the pair $\{a,b\}$ belongs to.
\end{definition}
\begin{proposition}[\cite{C-92},\S5] \label{1.19}
Let $(X,\mathcal{B},\in )$ be a BIBD with parameters $v$, $b$, $k$, $r$, $\lambda$. Then $vr=bk$ and $\lambda (v-1)=r(k-1)$.\\
\noindent If $(N,+,\cdot )$ is a finite planar nearring, then $(N,\mathcal{B}^*,\in )$ is a BIBD with parameters $v=|N|$, $k=|N^*/\equiv_m|$, $b=v(v-1)/k$, $r=v-1$, $\lambda=k-1$.
Moreover if the BIBD is circular, then $k\leq (3+\sqrt{4v-7})/2$ and this limit is effective.
\end{proposition}

\subsection{The family of circles $E^r_c$}

\begin{definition} \label{1.23}
Let $(N,\Phi)$ be a finite circular Ferrero pair, then for every $r,c\in N^0$ define the \emph{family of circles $E^r_c$} by
\bd
    E^r_c=\{\Phi(r)+b\mid b\in \Phi(c)\}.
\ed
Obviously $|E^r_c|=|\Phi(c)|$ (remember that $E^r_c\subset \mathcal{P}(N)$).
\end{definition}
We give now without proof some well known properties of the family of circles, for details see (\cite{K-92},\S 4) and (\cite{C-92},\S 6).
\begin{lemma} \label{1.24}
Let $r,r',c,c'\in N^0$. Then $E^r_c=E^{r'}_{c'}$ if and only if $\Phi(r)=\Phi(r')$ and $\Phi(c)=\Phi(c')$.
\end{lemma}
\begin{lemma}\label{1.25}
Let $A=\Phi(r)+a$ and $B=\Phi(r)+b$ be in $E^r_c$. Then there exists $\varphi\in \Phi$ such that $\varphi(A)=B$ and $\varphi(a)=b$. In particular if $\varphi\in \Phi$ exists such that $\varphi(a)=b$, then $\varphi(\Phi(r)+a)=\Phi(r)+b$ and so for any $\varphi\in \Phi$, we have $\varphi(E^r_c)=E^r_c$.
\end{lemma}
\begin{theorem} \label{1.26}
Let $A\in E^r_c$ intersect in only one point (i.e. it's tangent) exactly $m$ circles of $E^r_c$, and intersect exactly $n$ circles of $E^r_c$, each in two points. Then every $B\in E^r_c$ has this property. Moreover, if $A,B\in E^r_c$ and $r'\in N^0$, then $|A\cap \Phi(r')|=|B\cap \Phi(r')|$.
\end{theorem}

\section{Disks}

\begin{definition} \label{2.1}
Let $(N,\Phi)$ be a finite circular Ferrero pair and $\Phi(a)+b\in \mathcal{B}^*$,
then we define $\mathcal{D}(a;b)$, the \emph{disk of center $b$ and radius $a$}, as
\bd
    \mathcal{D}(a;b)=\{x\in \Phi(r)+c\mid r\neq 0,\ b\in \Phi(r)+c,\ |(\Phi(r)+c)\cap (\Phi(a)+b)|=1\} .
\ed
Now it is obvious to define the \emph{interior part of the circle} $\Phi(a)+b$, $\mathcal{I}(\Phi(a)+b)$, as
\bd
    \mathcal{I}(\Phi(a)+b)=\mathcal{D}(a;b)\setminus (\Phi(a)+b).
\ed
\end{definition}
This definition is coherent with what we know by Euclidean geometry, in particular the idea is to reconstruct the interior of a circle joining together all the circles tangent to it and containing its center, and this forces these circles to be  somehow internal at the given one. Moreover the definition is also inherent to the fixed nearring, in the sense that it does not require another structure to be used, and this is of course an improvement towards Clay's definition (\cite{C-92}, Def.(7.115)). Now we want to study the geometrical structure of the disks and in particular to give, at least in the more interesting cases, a very fast way to construct them and to check the membership of a point.

\subsection{Geometrical considerations}

First we prove that under some  enough general conditions, the definition of interior part is stable under translations and dilatations, and then we
investigate the cases in which this interior part is surely nonempty.\\
To simplify the notations, we write for short $\sqcup E^r_c$ to mean the set of points belonging to the circles of $E^r_c$, i.e.
\[
    \sqcup E^r_c=\{x\in \Phi(r)+b\mid \Phi(r)+b\in E^r_c\}.
\]
\begin{lemma} \label{2.3}
Let $(N,\Phi)$ be a circular Ferrero pair and $a,b\in N$ with $a\neq 0$, then
\[
    \mathcal{D}(a;b)=\mathcal{D}(a;0)+b
\]
and
\[
    \mathcal{I}(\Phi(a)+b)=\mathcal{I}(\Phi(a))+b,
\]
that is, the definitions of disk and of interior part are homogeneous respect to translation.
\end{lemma}
\begin{proof}
It is enough to observe, remembering that we always consider $r\neq 0$, that
\[
    \begin{array}{ll}
        \mathcal{D}(a;0)+b=\\
        =\{x\in \Phi(r)+c\mid 0\in \Phi(r)+c,\ |(\Phi(r)+c)\cap (\Phi(a))|=1\}+b=\\
        =\{x\in \Phi(r)+\overline{c}-b\mid 0\in \Phi(r)+\overline{c}-b,\ |(\Phi(r)+\overline{c}-b)\cap (\Phi(a))|=1\}+b=\\
        =\{x+b\in \Phi(r)+\overline{c}\mid b\in \Phi(r)+\overline{c},\ |(\Phi(r)+\overline{c})\cap (\Phi(a)+b)|=1\}+b=\\
        =\{x\in \Phi(r)+\overline{c}\mid b\in \Phi(r)+\overline{c},\ |(\Phi(r)+\overline{c})\cap (\Phi(a)+b)|=1\}-b+b=\\
        =\mathcal{D}(a;b).
    \end{array}
\]
For the interior part, we have
\[
    \begin{split}
        \mathcal{I}(\Phi(a)+b)&=\mathcal{D}(a;b)\setminus (\Phi(a)+b)=(\mathcal{D}(a;0)+b)\setminus (\Phi(a)+b)=\\
        &=(\mathcal{D}(a;0)\setminus \Phi(a))+b=\mathcal{I}(\Phi(a))+b.
    \end{split}
\]
\end{proof}
\begin{lemma} \label{2.4}
Let $(N,\Phi)$ be a circular field-generated Ferrero pair and let $(N,+,\cdot)$ be the generating field. If $a\in N^0$, then
\[
    \mathcal{D}(a;0)=a\cdot \mathcal{D}(1;0)
\]
and
\[
    \mathcal{I}(\Phi(a))=a\cdot \mathcal{I}(\Phi(1)),
\]
that is, the definitions of disk and of interior part are homogeneous respect to
dilatation.
\end{lemma}
\begin{proof}
Since $a$ acts as a f.p.f. automorphism, following the previous lemma, and observing that the multiplication symbol $\cdot$ in the field is always omitted, we have
\[
    \begin{array}{l}
        a\mathcal{D}(1;0)=\\
        =a\{x\in \Phi(r)+c\mid 0\in \Phi(r)+c,\ |(\Phi(r)+c)\cap \Phi(1)|=1\}=\\ =a\{x\in \Phi(r)+c\mid 0\in \Phi(r)+c,\ |(\Phi(a^{-1}ar)+a^{-1}ac)\cap \Phi(a^{-1}a)|=1\}=\\
        =a\{x\in \Phi(r)+c\mid 0\in \Phi(a^{-1}ar)+a^{-1}ac,\ |a^{-1}(\Phi(ar)+ac)\cap a^{-1}\Phi(a)|=1\}=\\
        =a\{x\in \Phi(a^{-1}ar)+a^{-1}ac\mid 0\in \Phi(ar)+ ac,\ |a^{-1}\{(\Phi(ar)+ac)\cap \Phi(a)\}|=1\}=\\
        =a\{x\in a^{-1}(\Phi(ar)+ac)\mid 0\in \Phi(ar)+ac,\ |(\Phi(ar)+ac)\cap \Phi(a)|=1\}=\\
        =\{ax\in \Phi(ar)+ac \mid 0\in \Phi(ar)+ ac,\ |(\Phi(ar)+ac)\cap\Phi(a)|=1\}=\\
        =\{y\in \Phi(\overline{r})+\overline{c} \mid 0\in \Phi(\overline{r})+ \overline{c},\ |(\Phi(\overline{r})+\overline{c})\cap \Phi(a)|=1\}=\\
        =\mathcal{D}(a;0),
    \end{array}
\]
and so the first equality. For the second we simply observe that
\[
    \begin{split}
        \mathcal{I}(\Phi(a))&=\mathcal{D}(a;0)\setminus \Phi(a)=a \mathcal{D}(1;0)\setminus a\Phi(1)=\\
        &=a(\mathcal{D}(1;0)\setminus \Phi(1))=a\mathcal{I}(\Phi(1)).
    \end{split}
\]
\end{proof}
By the two previous lemmas we immediately have the following corollary.
\begin{corollary}\label{2.5}
With the same hypothesis of the previous lemma, we have
\[
    \mathcal{D}(a;b)=a \cdot \mathcal{D}(1;0)+b=a\cdot \mathcal{D}(1;a^{-1}\cdot b)
\]
and
\[
    \mathcal{I}(\Phi(a)+b))=a\cdot \mathcal{I}(\Phi(1))+b=a\cdot \mathcal{I}(\Phi(1; a^{-1} \cdot b)).
\]
\end{corollary}
We have just proved that there is a simple method to construct every disk, namely to construct the disk $\mathcal{D}(1;0)$, in particular in the following we will show how to further simplify this construction. Besides, we note that we can prove properties of disks that are kept for (field) multiplication and sum, only speaking about one disk; in particular it is now obvious that all disks have the same cardinality, that is what we will prove in Theorem \ref{card} under some assumptions.

\begin{theorem} \label{2.7}
Let $(N,\Phi)$ be a finite circular field-generated Ferrero pair and let $|\Phi|$ be even, then, for each $a,b\in N$, with $a\neq 0$,
\[
    \mathcal{D}(a;b)=\sqcup E^{2^{-1}a}_{2^{-1}a}+b.
\]
\end{theorem}
\begin{proof}
At first we observe that $\Phi$ is a cyclic group of even order, so $-1\in \Phi$. Since $(N,\Phi)$ is a Ferrero pair, from Definition \ref{1.11} and $N$ finite, it follows that $-(-1)+1=1+1=2:N\rightarrow N$, $2:x\rightarrow x+x$ is a f.p.f. automorphism and so we can consider $2^{-1}$ (as automorphism). By the previous remark, we can consider, without loss of generality, only the disks like $\mathcal{D}(2a;0)$, $a\neq 0$.\\
Suppose that $|\Phi|=2n$, let $\varphi$ be a generator of $\Phi$ and $c\in \Phi(a)$, we now show that $\Phi(a)+c\subseteq \mathcal{D}(2a;0)$, or equivalently that
\[
    0\in \Phi(a)+c
\]
and
\[
    |(\Phi(a)+c) \cap \Phi(2a)|=1.
\]
Since $c\in \Phi(a)$, there exists $1\leq l\leq 2n$ so that $c=\varphi^l(a)$; observing that of course $\varphi^n=-1$, we get
\[
    0=\varphi^{l+n}(a)+c\in \Phi(a)+c
\]
and
\[
    \varphi^{l}(2a)=2\varphi^l(a)=\varphi^l(a)+\varphi^l(a)\in \Phi(a)+c
\]
and finally
\[
    \varphi^l(2a)\in ((\Phi(a)+c) \cap \Phi(2a)).
\]
Now we must prove $|(\Phi(a)+c) \cap \Phi(2a)|=1$. Without loss of generality, we suppose that $l=0$ (i.e., $c=a$), in such way our thesis becomes
$(\Phi(a)+a)\cap \Phi(2a)=\{2a\}$.\\
By contradiction let $m\not\equiv 0 \pmod{2n}$ such that $\varphi^m(2a)\in \Phi(a)+a$. Then there exists $1\leq h\leq 2n-1$ so that $\varphi^m(2a)=\varphi^h(a)+a$, or equivalently $\varphi^{-h}(\varphi^m(2a))=\varphi^{-h}(\varphi^h(a)+a)$, which means that
\bd
    \varphi^{m-h}(2a)=\varphi^{-h}(a)+a.
\ed
Then also $\varphi^{m-h}(2a)\in \Phi(a)+a$ and such element is different from $2a$ and $\varphi^m(2a)$ because every element of $\Phi$ is f.p.f.
 Therefore $|(\Phi(a)+a)\cap \Phi(2a)|\geq 3$ and this is a contradiction by the hypothesis of circularity. So we have proved that for every $c\in \Phi(a)$, it is $\Phi(a)+c\subseteq \mathcal{D}(2a;0)$, that means
\bd
    \sqcup E^a_a\subseteq \mathcal{D}(2a;0).
\ed
Now let $f\in \mathcal{D}(2a;0)$, then $f$ belongs to a circle passing through $0$ and, without loss of generality, tangent to $\Phi(2a)$ in $2a$; therefore let $\Phi(s)+s$ be such a circle. If $\varphi^t(s)+s=2a$, then
\bd
    \varphi^{-t}(\varphi^t(s)+s)=s+\varphi^{-t}(s)=\varphi^{-t}(2a)\in (\Phi(s)+s)\cap \Phi(2a)=\{2a\}
\ed
for which $t=0$, and $2s=2a$, that is $s=a$. In conclusion we have proved that
\bd
    \mathcal{D}(2a;0)\subseteq \sqcup E^a_a,
\ed
and with the previous inclusion, we have the thesis.
\end{proof}
\begin{corollary} \label{2.9}
In the same hypothesis of the previous theorem, the center is interior to the circle, in particular every disk is nonempty.
\end{corollary}

\begin{lemma}\label{oss2}
In the same hypothesis of the previous theorem, if $\Phi(b)\neq\{0\}$ and $\Phi(b)\neq\Phi(2a)$, then for all $C \in E^{a}_{a}$ is
\[
|C\cap \Phi(b)|\in\{0,2\}.
\]
\end{lemma}

\begin{proof}
Let $C \in E^{a}_{a}$ and $b \in N$. Then $C=\Phi(a)+a'$, for some $a'\in\Phi(a)$, or, equivalently, $C=\Phi(r)+r$, for some $r\in \Phi(a)$. If $|C\cap \Phi(b)|\neq 0$, let $c \in C \cap \Phi(b)$, it follows $\Phi(c)=\Phi(b)$. As $c\in C$, let $\varphi\in\Phi$ such that $c=\varphi(r)+r$, then also $\varphi^{-1}(c)=\varphi^{-1}(\varphi(r)+r)=r+\varphi^{-1}(r)$ belongs to $C$. So $C\cap \Phi(b)=C\cap \Phi(c)=\{c, \varphi^{-1}(c)\}$, and $c=\varphi^{-1}(c)$ if, and only if, $c=0$ or $\varphi=1$, which is against the hypothesis as in the first case $\Phi(b)=0$ and in the second $\Phi(b)=\Phi(2a)$.
\end{proof}

\begin{theorem} \label{card}
In the same hypothesis of the previous theorem, if $c\in \mathcal{D}(a;b)$, then $\Phi(c)+b\subseteq \mathcal{D}(a;b)$. More precisely, if $|\Phi|=2n$, then every disk is union of $n+1$ circles (one of which degenerates), that is there are
$c_1,...,c_{n-1}\in N^0\setminus \Phi(a)$, with $\Phi(c_i)\neq \Phi(c_j)$ if $i\neq j$, so that
\bd
    \mathcal{D}(a;b)=(\Phi(0) \cup \Phi(c_1)\cup...\cup \Phi(c_{n-1})\cup \Phi(a))+b.
\ed
It follows that every disk has exactly $2n^2+1$ points.
\end{theorem}
\begin{proof}
Without loss of generality, we can consider $b=0$, then from Theorem \ref{2.7} there exists $r\in N$ such that $\mathcal{D}(a;0)=\sqcup E^{r}_{r}$. Let $c\in \mathcal{D}(a;0)$, then $C\in E^{r}_{r}$ exists so that $c\in C$. By Lemma \ref{1.25}, we have that for any $\varphi\in \Phi$, $\varphi(C)\in E^{r}_{r}$ and so $\varphi(c)\in \mathcal{D}(a;0)$. For the generality we have chosen $\varphi$, the first part of the thesis follows.

Now let $c \in \mathcal{D}(a;0)\setminus (\Phi(0) \cup \Phi(a))$, then there exists $C \in E^{r}_{r}$ such that $c \in C$, and from Lemma \ref{oss2} $|C\cap \Phi(c)|=2$; in particular the $2n-2$ elements in $C\setminus(\Phi(0)\cup\Phi(a))$ belong two by two to $n-1$ distinct orbits. Let $\Phi(c_1),\Phi(c_2),\dots,\Phi(c_{n-1})$ be such orbits. As $c_i \in C$ and $C\subseteq \mathcal{D}(a;0)$, then, from the first part of the theorem, for all $i \in \{1,\dots,n-1\}$, $\Phi(c_i) \subseteq \mathcal{D}(a;0)$, and then $\Phi(0) \cup \Phi(c_1) \cup \dots \cup \Phi(c_{n-1}) \cup \Phi(a) \subseteq \mathcal{D}(a;0)$.

We now prove the other inclusion. Let $x \in \mathcal{D}(a;0)\setminus (\Phi(0) \cup \Phi(a))$, then $x\in C'$ for some $C'\in E^{r}_{r}$. Then from Lemma \ref{oss2} $|C'\cap\Phi(x)|=2$, moreover there exists $\varphi\in \Phi$ such that $C=\varphi(C')$, and then $|C\cap\Phi(x)|=|\varphi(C)\cap\Phi(x)|=2$. So let $c\in C\cap\Phi(x)$, then there exists $i\in \{1,\dots ,n-1\}$ such that $c\in\Phi(c_i)$, and from $c\in \Phi(x)$, it follows $\Phi(x)=\Phi(c_i)$, and then $x\in\Phi(c_i)$.

The last part is now an obvious corollary.
\end{proof}

\begin{corollary}
In the same hypothesis of the previous theorem, let $\varphi$ a generator of $\Phi$ and $r=2^{-1}a$, then
\[
\mathcal{D}(a;0)=\Phi(0)\cup\Phi(1)\cup\Phi((\varphi+1)r)\cup\dots\cup\Phi((\varphi^{n-1}+1)r).
\]
\end{corollary}

\begin{proof}
From the previous theorem we know that
\begin{equation}\label{eq1}
\mathcal{D}(a;0)=\Phi(0)\cup\Phi(1)\cup\Phi(c_1)\cup\dots\cup\Phi(c_{n-1}),
\end{equation}
where $c_1\dots c_{n-1} \in C=\Phi(r)+r$, and $\Phi(c_i)\neq \Phi(c_j)$ for $i \neq j$.
Now, let $c_i=\varphi^{n+j}+r$, for $j\in\{n,\dots,2n-1\}$, then $c=\varphi^{n+j}(1+\varphi^{n-j})(r)\in \Phi((\varphi^{n-j}+1)(r))$. So, if $c_i'=(\varphi^{n-j}+1)(r)$, we can replace $c_i$ by $c_i'$ in Equation \ref{eq1}, as $\Phi(c_i)=\Phi(c_i')$. So we can always suppose $c_i=(\varphi^{i}+1)r$ for $i \in \{1,\dots,n-1\}$.
\end{proof}

\subsection{Combinatorial properties}
We are now able to prove, in a special case, an important combinatorial property of the incidence structures obtained considering the sets of all disks, namely
\bd
    \mathcal{B}^{\mathcal{D}}=\{\mathcal{D}(a;b)\mid a,b\in N,\ a\neq 0\}=\{\sqcup(E^r_r+d)\mid r,d\in N,\ r\neq 0\}.
\ed

\begin{lemma}\label{lemma1}
Let $(N,\Phi)$ be a finite circular field-generated Ferrero pair with $|N|=p$ and $p$ prime. Then, for all $a,b,b'\in N$,
\[
\mathcal{D}(a;b)=\mathcal{D}(a;b') \Leftrightarrow b=b' \text{ or } \mathcal{D}(a;b)=N.
\]
\end{lemma}

\begin{proof}
From the previous lemmas, we equivalently show that
\[
\mathcal{D}(1;0)=\mathcal{D}(1;0)+b \Leftrightarrow b=0 \text{ or } \mathcal{D}(1;0)=N.
\]
Let $\mathcal{D}(1;0)=\mathcal{D}(1;0)+b$ with $b\neq 0$. Then $b,2b,\dots ,pb\in \mathcal{D}(1;0)$, and then $\langle b\rangle=N\subseteq \mathcal{D}(1;0)$, i.e., $\mathcal{D}(1;0)=N$.
The other implication is obvious.
\end{proof}

\begin{lemma}\label{lemma2}
Let $(N,\Phi)$ be a finite circular field-generated Ferrero pair with $|N|=p$ and $p$ prime, and let $|\Phi|$ be even. Then, for all $a,a',b,b'\in N$
\[
\mathcal{D}(a;b)=\mathcal{D}(a';b') \Leftrightarrow b=b' \text{ and } \mathcal{D}(a;0)=\mathcal{D}(a;0).
\]
\end{lemma}
\begin{proof}
From the previous lemmas, we equivalently show that
\[
\mathcal{D}(a;0)=\mathcal{D}(a';0)+b \Leftrightarrow c=0 \text{ and } \mathcal{D}(a;0)=\mathcal{D}(a',0).
\]
Let $\mathcal{D}(a;0)=\mathcal{D}(a';0)+b$, and $\varphi \in \Phi$, $\varphi \neq 1$. Then, from Lemma \ref{1.25} and Theorem \ref{2.7}, it follows that $\varphi(\mathcal{D}(a;0))=\varphi(\mathcal{D}(a';0)+b)=\varphi(\mathcal{D}(a';0))+\varphi(b)=\mathcal{D}(a';0)+\varphi(b)$, and $\mathcal{D}(a;0)=\varphi(\mathcal{D}(a;0))$. In particular
\[
\mathcal{D}(a';0)+b=\mathcal{D}(a';0)+\varphi(b),
\]
and applying Lemma \ref{lemma1}, $b=\varphi(b)$. As $\varphi\neq 1$, and $\Phi$ is f.p.f., it follows $c=0$.
The other implication is obvious.
\end{proof}

\begin{lemma}\label{lemma3}
Let $(N,\Phi)$ be a finite circular field-generated Ferrero pair with $|\Phi|$ even, and suppose $\mathcal{D}(1;0)\setminus \{0 \}$ is not a multiplicative group. Then, for all $a,a'\in N$, $\mathcal{D}(a;0)=\mathcal{D}(a';0)$ if, and only if, $a\in \Phi(a')$.
\end{lemma}
\begin{proof}
From the previous lemmas, we equivalently show that
\[
\mathcal{D}(a;0)=\mathcal{D}(1;0) \Leftrightarrow a\in \Phi(1).
\]
Suppose $\Phi=2n$, from Theorem \ref{card}, there exists $c_1,\dots,c_{n-1}$ such that $\mathcal{D}(1;0)=\Phi(0)\cup\Phi(1)\cup\Phi(c_1)\cup\dots\cup\Phi(c_{n-1})$.

Let now $a\in \Phi(1)$. From Lemma \ref{2.4},
$\mathcal{D}(a;0)=\Phi(0)\cup\Phi(a)\cup\Phi(ac_1)\cup\dots\cup\Phi(ac_{n-1})$, and then $\mathcal{D}(a;0)=\mathcal{D}(1;0)$, as $a\in \Phi(1)$.

Conversely, let $\mathcal{D}(a;0)=\mathcal{D}(1;0)$ and suppose $a\notin \Phi(1)$.
Observe that for all $1\leq i,j\leq n-1$ with $i\neq j$, $a^i$ and $a^j$ belongs to different orbits, and then it is possible to rename $c_1,\dots ,c_{n-1}$ such that $a^i\in \Phi(c_i)$.
Moreover, it follows that $a^n$ is the smallest power of $a$ belonging to $\Phi(1)$.
Let now $t$ be the order of $a^n$ in the cyclic group $\Phi(1)$, which is isomorphic to a subgroup of $\Z_p^*$. By definition $(a^n)^t=(a^t)^n=1$, and then $a^t$ has order $n$ in $\Z_p^*$, which implies $a^t\in \Phi(1)$ and $n\leq t\leq 2n$. Let now $q\in \mathbb{N}$ and $0\leq r<n$ be such that $t=nq+r$, from $a^t=(a^n)^qa^r$ it follows $a^r=a^t(a^n)^{-q}\in \Phi(1)$, and then $r=0$ and $q\in\{1,2\}$.
If $q=2$, then $t=2n$ and all the $2n^2$ distinct power of $a$ belong to $\mathcal{D}(1;0)$, and so $\mathcal{D}(1;0)\setminus\{0\}$ is the multiplicative subgroup generated by $a$, an absurd.
So let $q=1$ and $t=n$. If $n$ is even, then $(a^{n/2})^2n=1$ and $a^{n/2}\in \Phi(1)$ with $n/2<n$, which is absurd. If $n$ is odd, then $\Phi=\langle-a^n\rangle$ and $\langle-a\rangle\subseteq \mathcal{D}(1;0)$, and again $\mathcal{D}(1;0)\setminus\{0\}$ is the multiplicative subgroup generated by $-a$, an absurd.
\end{proof}

\begin{theorem} \label{2.20}
Let $(N,\Phi)$ be a finite circular field-generated Ferrero pair with $|N|=p$, $p$ prime, and $|\Phi|=2n$. Then $(N, \mathcal{B}^{\mathcal{D}}, \in)$ is a BIBD. Moreover, if $\mathcal{D}(1;0)\setminus\{0\}$ is not a multiplicative group, then $(N, \mathcal{B}^{\mathcal{D}},\in)$ is a BIBD of parameters $v=p$, $b=p(p-1)/2n$, $k=2n^2+1$, $r=(p-1)(2n^2+1)/2n$, $\lambda=n(2n^2+1)$.
\end{theorem}
\begin{proof}
From Theorem \ref{card} each block contains exactly $2n^2+1$ elements, so $k=2n^2+1$. Let now $b\in N$, then $b\in \mathcal{D}(a;b)+c\Leftrightarrow b\in \mathcal{D}(a;0)+b+c\Leftrightarrow -c\in \mathcal{D}(a;0)$, and this is possible for $2n^2+1$ distinct values of $c$. From Lemma \ref{lemma1} there exist exactly $2n^2+1$ distinct disks of radius $a$ containing $b$, and from Lemma \ref{lemma2} if there exist $d$ distinct disks of radius $0$, then $b$ belongs to $d(2n^2+1)$ distinct blocks. For the generality we have chosen $b$, we have $r=d(2n^2+1)$. Now let $x,y\in N$, $x\neq y$ and $x\neq 0$; then
\[
    x,y\in \mathcal{D}(a;b) \,\Leftrightarrow\, x-x,y-x\in \mathcal{D}(a;b)-x \,\Leftrightarrow\, 0,y-x\in \mathcal{D}(a;b-x),
\]
and therefore $\means{x,y}=\means{0,y-x}$. Now if $z\in N\setminus\{0,1\}$, we obtain
\[
    0,z\in \mathcal{D}(a;b) \,\Leftrightarrow\, 0z^{-1},zz^{-1}\in \mathcal{D}(a;b)z^{-1} \,\Leftrightarrow\, 0,1\in \mathcal{D}(az^{-1};bz^{-1}),
\]
that is, from Lemma \ref{lemma2}, $\means{0,z}=\means{0,1}$ (indeed, if $\mathcal{D}(c;d)\neq \mathcal{D}(a;b)$ contains $\{0,z\}$, then the disk $\mathcal{D}(z^{-1}c;z^{-1}d)\neq \mathcal{D}(z^{-1}a;z^{-1}b)$ contains $\{0,1\}$). Therefore $\means{x,y}=\means{0,y-x}=\means{0,1}$ and $(N,\mathcal{B}^{\mathcal{D}},\in)$ is a BIBD.
We know from Lemma \ref{lemma1} that $\mathcal{D}(a;b)=\mathcal{D}(a;c)$ implies $b=c$, so the distinct disks of fixed radius are as many as the number of elements of $\Z_p$. From Lemma \ref{lemma3} two disks having the same centre coincide if, and only if, their radii belong to the same orbit, and then the number of distinct disks with fixed centre is exactly the number of the later classes of $\Phi$ in $\Z_p$. Then the number of distinct disks is $p(p-1)/2n$, and knowing that  $(N,\mathcal{B}^{\mathcal{D}},\in)$ is a BIBD, the values of $r$ and $\lambda$ can be obtained from those of $v$ ,$b$, and $k$ applying Proposition \ref{1.19}.
\end{proof}

\begin{remark}
In the hypothesis of Theorem \ref{card}, if $|N|=p$, $p$ prime, the BIBD $(N, \mathcal{B}^{\mathcal{D}}, \in)$ can not contain only one disk, indeed if $\mathcal{D}(a;b)=N$, then, from Theorem \ref{card} $2n^2+1=p$, but from circularity $2n\leq {3+\sqrt{4p-7}\over 2}$ and so either $n=1$, which is not possible as the nearring is planar, or $n=2$ and $p=9$, which is not possible as $p$ is prime.
\end{remark}

\section{Future work}
In this paper we have obtained some quantitative results in the even case ($|\Phi|$ even), leaving open the odd case. This is not very surprising, in fact it seems that the odd case is always harder than the even one, see for example (\cite{K-92},\S IV.5.19). Nevertheless, we are able to give some reasonable conjectures, in particular what we believe is that if (with the usual
notation) $|\Phi|=2n+1$, then $|\mathrm{M}^{a,b}|=2n$. From this, we could easily deduce that in the field-generated case the formula
\bd
    \mathcal{D}(a;b)=\bigcup_{r\in \mathrm{M}^{a,b}}\{x\in \Phi(r)+c\mid \Phi(r)+c\in E^r_{-r}+b\}
\ed
holds.

\newpage
\noindent Anna Benini\\
Dipartimento di Matematica,\\
Universit\`{a} degli Studi di Brescia,\\
via Valotti, 9 \\
25133 Brescia, Italy \\
\verb"anna.benini@ing.unibs.it"

\smallskip
\smallskip

\noindent Achille Frigeri, Ph.D. \\
Dipartimento di Elettronica e Informazione, \\
Politecnico di Milano, \\
Via Ponzio 34/5, \\
20133 Milano, Italy\\
\verb"achille.frigeri@polimi.it"

\smallskip
\smallskip

\noindent Fiorenza Morini\\
Dipartimento di Matematica,\\
Universit\`{a} degli Studi di Parma,\\
viale G.P. Usberti, 53/A \\
43100 Parma, Italy \\
\verb"fiorenza.morini@unipr.it"

\end{document}